\documentclass[11pt,a4paper,reqno]{article}
\usepackage{amsthm, mathrsfs,amssymb,amsmath}
\usepackage{amsfonts}
\usepackage{enumerate}
\usepackage{mathtools}
\usepackage{authblk}

\allowdisplaybreaks
\numberwithin{equation}{section}

\usepackage{graphicx}
\usepackage{enumitem}

\newtheorem{theorem}{Theorem}[section]
\newtheorem{lemma}[theorem]{Lemma}
\usepackage{xcolor}

\theoremstyle{definition}
\newtheorem{definition}[theorem]{Definition}
\newtheorem{example}[theorem]{Example}
\newtheorem{proposition}[theorem]{Proposition}
\newtheorem{corollary}[theorem]{Corollary}
\newtheorem{note}[theorem]{Note}

\theoremstyle{remark}
\newtheorem{remark}[theorem]{Remark}

\numberwithin{equation}{section}

\usepackage{authblk}

\title{Further analysis of Multivariate fractal functions}
\author[1*]{Amit}
\author[2]{Vineeta Basotia}
\author[3]{Ajay Prajapati}

\affil[1*,2]{Department of Mathematics, JJT University, Jhunjhunu, Rajasthan 333001, India }	

\affil[3]{Department of Computer Science, Banaras Hindu University, Varanasi 221005, India}
\affil[1*]{Corresponding author: amit60363@gmail.com}
\affil[2]{vm.jjtu@gmail.com}
\affil[3]{ajaypraja640@gmail.com}
\begin{document}
	
	
	
\date{}	
	
	

	
\maketitle







 

\begin{abstract} 
 The aim of this paper is to characterize a fractal operator associated with multivariate fractal interpolation functions (FIFs) and study the several properties of this fractal operator. Further, with the help of this operator, we characterize a latest category of functions and study their approximation aspects. The basic characteristics of this multivariate fractal operator's are given in several ways in this note. The extension of this fractal operator to the $L^p$-spaces for $ p \ge 1$ are also examined. Multivariate continuous fractal functions approximation characteristics are also examined.
\end{abstract}
{\bf Keywords :} {Fractal interpolation, Fractal operator, Bounded linear operator}\newline
{\bf Math subject Classification:} {Primary 28A80; Secondary 47H10, 54E50}

\section{Introduction} 
Hutchinson\cite{Hu} suggested the idea of an iterated function system (IFS), which is generalized by the Barnsley\cite{MB}. IFSs give a potential approach to study the structure of images and serve as a standard foundation for expressing self-referential sets like fractals. Numerous modifications to the traditional IFS like Partitioned IFS, Recurrent IFS and Super IFS etc are explored by many authors in literature \cite{MFB1, MFB2, YF}. Agrawal and Som\cite{EPJST} examined the Sierpinski gasket-based FIF and its box dimension corresponding to a continuous function. Further, they also studied the approximation of functions by fractal functions corresponding to the $L^p$-norm on the Sierpi\'nski gasket in \cite{RIM}. Sahu and Priyadrshi \cite{SP} demonstrated the box dimension of the graph of a harmonic function on Sierpi\'nski  gasket. Verma and Sahu \cite{VFS} calculated the fractal dimension of a continuous function with a bounded variation on  Sierpi\'nski  gasket.   In this order, Agrawal and Som  \cite{Vishal} given the concept of dimension preserving approximation over the rectangular domain for the bivariate continuous function. Verma and Viswanatha \cite{sverma5} confirmed the presence of the fractal rational trigonometric approximant to continuous function with real-values on a compact interval and develop Weierstrass type approximation theorems. We can see some recent work done by many authors on fractal  dimensions in literature\cite{JVC1,Jha,Liang,VL1,sverma3}.
\par

The partial integrals and derivatives of the bivariate FIFs studied by the Chandra and Abbas\cite{SS1}. They also deal with the integral transforms and its fractional order of bivariate FIFs. Chandra and Abbas\cite{SS4} investigated the mixed Weyl–Marchaud fractional derivative and also show its bivariate FIFs is still bivariate FIFs. Chandra and Abbas\cite{SS2} studied the fractal dimension of the graph of the mixed Riemann-Liouville fractional integral for different choices of continuous functions over a rectangular region. Chandra and Abbas\cite{SS3} worked on box dimension of the mixed Katugampola fractional integral of two-dimensional continuous functions defined on $[0,1]\times[0,1]$.
Study of fractal dimensions of fraction functions on function spaces is given in \cite{SS5}.
\par
Banach contraction principle is an important part of fixed point theory. Banach contraction principle is modified in its generalizations where Lipschitz constant is replaced by some real valued function in \cite{SR1,Amit}. Boyd and Wong \cite{PV1} also given the detailed study about its generalizations and showed fixed points' existence.  The construction of FIFs with the use of Rakotch and Matkowski fixed point theorems can see in \cite{SR2}. The theory on fractal dimension of vector valued function given by the Verma et al. \cite{Ver1}.
\par
 Verma \cite{VM1} investigated approximation issues in relation to the fractal dimensions of function as well as their derivative.
S. Jha et al.\cite{Jha2} explored how well non-stationary fractal polynomials approximate continuous functions and the sequence of mappings of non-stationary IFSs were used here. Priyadarshi and his group \cite{Ver71} studed fractal dimensions of attractors of  bi-Lipschitz IFSs and also obtained some results on quantization dimension.
Massopust et al. \cite{Mas1} summerized some results from the theory of fractal functions. The class of fractal functions under consideration is generated by IFSs. The same author in \cite{Mas2} discussed the theory of local IFS, local fractal functions, fractal surfaces, and their relation to wavelets as well as wavelet sets.
Navascues \cite{M2} developed a method to compute the interpolation resulting from the coupling of two functions of different types. The $\alpha$-fractal function for $f$ is the name of the FIF connected to an IFS.The interaction between fractal functions and other traditional disciplines of mathematics, such as operator theory, complex analysis, harmonic analysis, and approximation theory, is made possible by the operator formulation of fractal functions that is partially concealed in the creation of FIFs  \cite{M2,M1,Mnew,Mnew2}. To know more about FIFs, read this article \cite{PV11}.
 V. Agrawal et al. \cite{VSV1} discussed the  continuous dependence for bivariate  fractal function on different parameters and in \cite{PAS3}  they constructed the multivariate FIFs and defined the $\alpha$-fractal function corresponding to the multivariate continuous function which defined on $[0,1]^q$. M. Pandey et al. \cite{PSV2} presented the idea of the fractal approximation for set-valued continuous map and the $\alpha$ -fractal functions defined on bounded as well as closed interval. M. Verma et al.\cite{Ver12} examined fractal dimension of complex-valued functions.
\par 
Our paper is organized as follows. In Section $2$, we recall the construction of multivariate $\alpha$-fractal function due to Megha, Vishal and Som following Ruan's construction. The fractal operator $\mathcal{F}_{\Delta, L}^\alpha: \mathcal{C}(I^q, \mathbb{R})   \rightarrow \mathcal{C}(I^q, \mathbb{R})$ will be defined and its basic features will be established in section 3 that converts a bivariate continuous function f to its fractal analogue ${f}_{\Delta, L}^\alpha$.
In Section 4, we extend our fractal operator to $L^p(I^q, \mathbb{C})$ using the standard density parameters. Some characteristics of multivaruate ``fractal polynomials" and continuous multivaruate $\alpha$-fractal functions that relate to approximation are examined in Section 5.

\section{Auxiliary Apparatus} 
\subsection{Multivariate Fractal Interpolation Function}\label{subsec1}
Here we revisit the construction, which is a generalization of Ruan's construction \cite{Ruan}. The data of interpolation is given as 
\[\left\{\left(x_{1,j_1},\ldots, x_{k,j_k}, z_{(1,j_1),\ldots,(k,j_k)}\right): j_1=0,1,\ldots, N_1; ~ \ldots ~; j_k=0,1,\ldots, N_k\right\},\]
such that $0=x_{q,0}<\cdots < x_{q,N_q}=1$ for each $I_q$, where $I_q$ denotes the $q^{th}$ unit interval in $I^k$.\\
We will now write some notations for our convenience as follows;\\
$ \Sigma_q=\{1,2,\ldots,q\},$ $  \Sigma_{q,0}=\{0,1,\ldots, q \},$ $\partial  \Sigma_{q,0}=\{0,q\} $ and int$ \Sigma_{q,0}=\{1,2,\ldots,q-1\}.$ \\ 
Further, a net $\Delta$ on $I^k$ is defined as follows:
\[\Delta:=\left\{(x_{1,j_1}
, \ldots , x_{k,j_k}) \in I ^k: j_q \in \Sigma_{N_q,0}, ~ 0= x_{q,0} <\ldots < x_{q,N_q} =1,~ q \in \Sigma_{k}\right\}.\]
For each $j_q \in  \Sigma_{N_{q}}$, let us define $I_{q,j_{q}} = [x_{q,j_{q}-1}, x_{q,j_{q}}]$ and contraction functions $v_{q,j_{q}}
: I_q \rightarrow I_{q,j_q}$ such that
\begin{equation}\label{affine}
    \begin{aligned}
        v_{q,j_q}(x_{q,0})=&x_{q,j_q-1}, \quad \quad v_{q,j_q}(x_{q,N_q})=x_{q,j_q}, \text{  if $j_q$ is odd,}\\
 v_{q,j_q}(x_{q,0})=&x_{q,j_q}, \quad \quad v_{q,j_q}(x_{q,N_q})=x_{q,j_q-1}, \text{ if $j_q$ is even, and }\\
 \left\lvert v_{q,j_q}(x)-v_{q,j_q}(y) \right\rvert &\leq \mu_{q,j_q}\lvert x-y\rvert \text{ for all } x,y \in I_q, \text{ and } 0< \mu_{q,j_q}<1.
    \end{aligned}
\end{equation}
It is easy to observe from the definition of $v_{q,j_q}$ that 
\[v_{q,j_{q}}^{-1}(x_{q,j_{q}})=v_{q,j_{q}+1}^{-1}(x_{q,j_q}) \text{ for all } j_q\in \text{int}\Sigma_{q,0} \text{ and } q\in \Sigma_k.\]
Define a function $\eta : \mathbb{Z}\times \{0, N_1, \ldots, N_k\}\rightarrow \mathbb{Z}$ by 
\begin{equation*}
    \begin{cases}
    \eta(j,0)=j-1, \quad \eta(j,N_1)=\cdots =\eta(j,N_k)=j, &\text{ when $j$ is odd,}\\
    \eta(j,0)=j, \quad \eta(j,N_1)=\cdots =\eta(j,N_k)=j-1, & \text{ when $j$ is even.}
    \end{cases}
\end{equation*}
 Set $K=I^k\times \mathbb{R}$. For each $(j_1,\ldots, j_k)\in \overset{k}{\underset{q=1}{\prod}} \Sigma_{N_q}$, let $F_{j_1,\ldots, j_k}:K\rightarrow \mathbb{R}$ be a continuous function satisfying
\[F_{j_1\cdots j_k}\left(x_{1,q_1},\ldots, x_{k,q_k}, z_{(1,q_1)\cdots(k,q_k)}\right)=z_{(1,\eta(j_1,q_1))\cdots (k,\eta(j_k,q_k))} \]
for all $ (q_1,\ldots, q_k)\in \prod_{i=1}^{k}\partial \Sigma_{N_i,0},$ and 
\[\left\lvert F_{j_1\cdots j_k}(x_1,\ldots,x_k, z^{*})-F(x_1,\ldots, x_k, z^{**})\right \rvert \leq \gamma_{j_1\cdots j_k}\lvert z^{*}-z^{**}\rvert,
\]
 for all $(x_1,\ldots, x_k)\in I^{k},~ z^{*}, z^{**}\in \mathbb{R},$
 and  $0<\gamma_{j_1\cdots j_k}<1~$ is any given constant.\\
For each $(j_1,\ldots, j_k)\in \overset{k}{\underset{q=1}{ \prod}} \Sigma_{N_q}$, define a function $\Omega_{j_1\cdots j_k}:K\rightarrow \overset{k}{\underset{q=1}{ \prod}}I_{q,j_q}\times \mathbb{R}$ such that
\begin{equation}\label{IFS}
  \Omega_{j_1\cdots j_k}(x_1,\ldots, x_k, z)=\left(v_{1,j_1}(x_1),\ldots, v_{k,j_k}(x_k),F_{j_1\cdots j_k}(x_1,\ldots, x_k,z)\right).  
\end{equation}
Then, with the assumptions on $F_{j_1\cdots j_k}$ and $v_{q,j_q}$, for each $q\in  \Sigma_k$, one can prove that $\Omega_{j_1\cdots j_k}$ is a contraction function, hence $$\left\{K,~\Omega_{j_1\cdots j_k}: (j_1,\ldots,j_k)\in \overset{k}{\underset{q=1}{ \prod}} \Sigma_{N_q}\right\}~$$ is an IFS.

The next theorem can be accomplished by using the technique similar to \cite{Ruan}.
\begin{theorem}\label{matching}
Consider $\left\{K,\Omega_{j_1\cdots j_k}~: (j_1,\ldots,j_k)\in \overset{k}{\underset{q=1}{ \prod}} \Sigma_{N_q}\right\}$ be the IFS as defined in \eqref{IFS}. For every $q=1,\ldots,k$, suppose $F_{j_1\cdots j_k}$ satisfies the following matching conditions: for all $j_q\in \text{int} \Sigma_{N_{q},0}, i\neq q,~j_i\in  \Sigma_{N_i} \text{ and } x_{q}^{*}=v_{q,j_q}^{-1}(x_{q,j_q})=v_{q,j_q+1}^{-1}(x_{q,j_q}),$
\begin{equation}
    F_{j_1\cdots j_k}(x_1,\ldots, x_{q}^{*}, x_{q+1},\ldots, x_k, z)=F_{j_1\cdots j_{q-1},j_{q}+1 j_{q+1}\cdots j_k}(x_1,\ldots, x_{q}^{*}, x_{q+1},\ldots, x_k, z),
\end{equation}
for all $x_i\in I_i, ~ z\in \mathbb{R}$. Then, there is a unique continuous function $\mathcal{A}: I^{k}\rightarrow \mathbb{R}$ such that $\mathcal{A}(x_{j_1},\ldots, x_{j_k})=z_{j_1\cdots j_k}$ for all $(j_1,\ldots, j_k)\in \overset{k}{\underset{q=1}{ \prod}} \Sigma_{N_q,0}$ and $$H= \bigcup\left\{\Omega_{j_1\cdots j_k}(H): (j_1,\ldots,j_k)\in \overset{k}{\underset{q=1}{ \prod}} \Sigma_{N_q}\right\},$$ where $H=\mathcal{G}(\mathcal{A})$.
We pronounce $H$ the FIS and $\mathcal{A}$ the FIF with respect to the IFS defined in \eqref{IFS}.
\end{theorem}

\subsection{Multivariate Fractal Interpolation surfaces}\label{subsec2}
Define $F_{j_1\cdots j_k}: K\rightarrow \mathbb{R}$ such that 
\[F_{j_1\cdots j_k}(x_1,\ldots, x_k, z)=\delta z+\mathcal{B}_{j_1\cdots j_k}(x_1,\ldots,x_k),\]
where $\mathcal{B}_{j_1\cdots j_k}(x_{1,q_1},\ldots x_{k,q_k})=z_{(1,\eta(j_1,q_1)),\cdots, (k,\eta(j_k,q_k))}-\delta z_{(1,q_1)\cdots (k,q_k)}$ and $\lvert \delta \rvert<1.$
Define  $\Omega_{j_1\cdots j_k}:K\rightarrow \overset{k}{\underset{q=1}{ \prod}}I_{q,j_q}\times \mathbb{R}$  as
\[\Omega_{j_1\cdots j_k}(x_1,\ldots, x_k,z)=\left(v_{1,j_1}(x_1),\ldots, v_{k,j_k}(x_k), F_{j_1\cdots,j_k}(x_1,\ldots, x_k,z)\right).\]
Notice that $F_{j_1\cdots j_k}$ satisfies the conditions of Theorem \ref{matching}.
Now for all $(j_1,\ldots, j_k)\in \overset{k}{\underset{q=1}{ \prod}} \Sigma_{N_q}$, define Read-Bajraktarevi\'c (RB) operator, $T:\mathcal{C}(I^k)\rightarrow \mathcal{C}(I^k)$ such that
\[Tg(x_1,\ldots, x_k)=F_{j_1\cdots j_k}\left(v_{j,j_1}^{-1}(x_1),\ldots, v_{k,j_k}^{-1}(x_k), g\left(v_{1,j_1}^{-1}(x_1),\ldots,v_{k,j_k}^{-1}(x_k)\right)\right).\]
Then, the multivariate FIF,  $\mathcal{A}$ will be the unique fixed point of $T$. Therefore, $\mathcal{A}$ will satisfy the following self-referential equation:
\[
    \mathcal{A}(x_1,\ldots, x_k)=F_{j_1\cdots j_k}\left(v_{j,j_1}^{-1}(x_1),\ldots, v_{k,j_k}^{-1}(x_k), \mathcal{A}\left(v_{1,j_1}^{-1}(x_1),\ldots,v_{k,j_k}^{-1}(x_k)\right)\right).
\]
Therefore, we have
\begin{equation}\label{interpolation}
    \mathcal{A}(x_1,\ldots, x_k)=\delta \mathcal{A}\left(v_{1,j_1}^{-1}(x_1),\ldots,v_{k,j_k}^{-1}(x_k)\right)+\mathcal{B}_{j_1\cdots j_k}\left(v_{1,j_1}^{-1}(x_1),\ldots,v_{k,j_k}^{-1}(x_k)\right),
\end{equation}
for all $(x_1,\ldots, x_k)\in \overset{k}{\underset{q=1}{ \prod}}I_{q,j_q}$ and $(j_1,\ldots, j_k)\in \overset{k}{\underset{q=1}{ \prod}}~ \Sigma_{N_q}.$

\begin{note}\label{Naffine}
For every $x\in I_{q,{j_{q}}}$, let us define
$v_{q,j_{q}}(x)=a_{q,j_{q}}(x)+b_{q,j_{q}}$ such that
\begin{align*}
    a_{q,j_{q}}&=\frac{x_{q,j_{q}}-x_{q,j_{q-1}}}{x_{q,N_{q}}-x_{q,0}} \quad \text{and} \quad b_{q,j_{q}}= \frac{x_{q,j_{q}}x_{q,N_{q}}-x_{q,j_{q-1}}x_{q,0}}{x_{q,N_{q}}-x_{q,0}}, \text{  if $j_q$ is odd}\\
    a_{q,j_{q}}&=\frac{x_{q,j_{q-1}}-x_{q,j_{q}}}{x_{q,N_{q}}-x_{q,0}} \quad \text{and} \quad b_{q,j_{q}}=\frac{x_{q,j_{q-1}}x_{q,N_{q}}-x_{q,j_{q}}x_{q,0}}{x_{q,N_{q}}-x_{q,0}}, \text{  if $j_q$ is even}.
\end{align*}
 Then, obviously $v_{q,j_{q}}$ satisfies Equation \eqref{affine}.
\end{note}

\section{Multivariate $\alpha$ - fractal functions and associated fractal operator on $\mathcal{C}(I^k , \mathbb{R})$} \label{FIF}
\subsection{$\alpha$- fractal function}
\hspace{0.9cm}Let $f \in \mathcal{C}(I^k)$, then using the notions defined in Subsection \ref{subsec1}, let us denote
\[Q_{j_1\ldots j_k}(x_1,\ldots,x_k):= \left(v_{1,j_1}^{-1}(x_1),\ldots,v_{k,i_k}^{-1}(x_k)\right),\] 
where $(x_1,\ldots,x_k) \in \overset{k}{\underset{q=1}{ \prod}}I_{q,j_q}.$\\
Let $\alpha \in  \mathcal{C}(I^k)$ be such that $\|\alpha\|_{\infty}<1.$ Assume further that $s \in  \mathcal{C}(I^k)$ satisfies $s(x_{1,i_1},\ldots,x_{k,i_k})=f(x_{1,i_1},\ldots,x_{k,i_k}),$ for all $(i_1, \ldots,i_k) \in
\overset{k}{\underset{q=1}{ \prod}} \partial  \Sigma_{N_q,0}.$ \\
Now define $\mathcal{H}_{j_{1}\dots j_{k}}: K \rightarrow \mathbb{R}$ by
\begin{align}
 &\mathcal{H}_{j_1 \dots j_k}(y_1, \dots, y_k, z) \nonumber \\
 =&f\left(v_{1,j_1}(y_1), \dots, v_{k,j_{k}}(y_k)\right)+\alpha\left(v_{1,j_1}(y_1), \dots, v_{k,j_{k}}(y_k)\right) \left(z-s(y_1, \dots, y_k)\right),\label{falpha}
 \end{align}
For each $(j_1, \dots, j_k) \in \overset{k}{\underset{q=1}{ \prod}}  \Sigma_{N_{q}}$, define $W_{j_1 \dots j_k}: K \rightarrow \overset{k}{\underset{q=1}{ \prod}}I_{q,j_q}\times \mathbb{R}$
\[
W_{j_1 \dots j_k}(x_1, \dots, x_k, z)=\left(v_{1,j_1}(x_1), \dots, v_{k,j_{k}}(x_k), \mathcal{H}_{j_1 \dots j_k}(x_1, \dots, x_k, z)\right).
\]
Now observe that $\mathcal{H}_{j_1 \dots j_k}$ satisfies the conditions of Theorem \ref{matching}, therefore we have the following:
 \begin{theorem} Let $\left\{K, W_{j_1 \dots j_k}:(j_1, \dots, j_k) \in \overset{k}{\underset{q=1}{ \prod}}  \Sigma_{N_{q}}\right\}$ be the IFS. Then there exists a unique continuous function $f^{\alpha}_{\Delta,s}: I^k \rightarrow$ $\mathbb{R}$ such that $f^{\alpha}_{\Delta,s}\left(x_{1,j_1},\dots,x_{k,j_k}\right)=f\left(x_{1,j_1},\dots,x_{k.j_k}\right)$ for all $(j_1, \dots, j_k) \in \overset{k}{\underset{q=1}{ \prod}} \Sigma_{ N_{q},0}$ and $G= \bigcup\left\{W_{j_1\cdots j_k}(G): (j_1,\ldots,j_k)\in \overset{k}{\underset{q=1}{ \prod}} \Sigma_{N_q,0}\right\}, $ where $\mathcal{G}(f^{\alpha}_{\Delta,s})=G$.
\end{theorem}

Now define a RB operator, $T:\mathcal{C}(I^k)\rightarrow \mathcal{C}(I^k)$ such that 
\[T(h)(x_1,\ldots, x_k)=\mathcal{H}_{j_1 \dots j_k}\left(v_{1,j_1}^{-1}(x_1),\dots,v_{k,j_k}^{-1}(x_k), h\left(v_{1,j_1}^{-1}(x_1),\ldots,v_{k,j_k}^{-1}(x_k)\right)\right),\]
then there exists a fixed point of this operator, $f^{\alpha}_{\Delta,s}$ known as multivariate variate $\alpha$-fractal function and it satisfies the following self-referential equation,
\[f^{\alpha}_{\Delta,s}(x_1, \dots, x_k)=\mathcal{H}_{j_1 \dots j_k}\left(v_{1,j_1}^{-1}(x_1),\dots,v_{k,j_k}^{-1}(x_k), f^{\alpha}_{\Delta,s}\left(v_{1,j_1}^{-1}(x_1),\ldots,v_{k,j_k}^{-1}(x_k)\right)\right),\]
for all $(x_1, \dots, x_k) \in \overset{k}{\underset{q=1}{ \prod}}I_{q,j_q}.$
That is,
 \begin{equation}\label{alphafrac}
 \begin{aligned}
   f^{\alpha}_{\Delta,s}(x_1,\ldots,x_k)
   =& f(x_1,\ldots,x_k)+\alpha(x_1,\ldots,x_k)\left[ f^{\alpha}_{\Delta,s}\left(v_{1,j_1}^{-1}(x_1),\ldots,v_{k,j_k}^{-1}(x_k)\right)\right.\\
   &-s\left.\left(v_{1,j_1}^{-1}(x_1),\ldots,v_{k,j_k}^{-1}(x_k)\right)\right],
   \end{aligned}
\end{equation}
for $(x_1,\dots,x_k) \in \overset{k}{\underset{q=1}{ \prod}} I_{q,j_q},~ j_q \in  \Sigma_{N_q},~q \in   \Sigma_k.$

\subsection{Fractal operator}
We now demonstrate the fractal operator associated with the notion of multivariate $\alpha$- fractal function. Let the parameters  $\Delta$, $\alpha$ and $s$ used in the constructions of multivariate $\alpha$- fractal function be fixed. Then, there exist an operator corresponding to the multivariate $\alpha$- fractal function
is represented by $\mathcal{F}^{\alpha}_{\Delta,s}: \mathcal{C}(I^k) \to \mathcal{C}(I^k)$, which results each fixed given/function to its fractal counterpart. That is,  
$$\mathcal{F}^{\alpha}_{\Delta,s}: \mathcal{C}(I^k) \to \mathcal{C}(I^k),~~\mathcal{F}^{\alpha}_{\Delta,s} (f) = f^{\alpha}_{\Delta,s}, ~~ \forall ~~f\in \mathcal{C}(I^k) $$ 
if $s=Df$, where $D:\mathcal{C}(I^k) \to \mathcal{C}(I^k) $ is a bounded linear operator, then the operator $\mathcal{F}^{\alpha}_{\Delta,D}$  is a bounded linear operator called as the fractal operator .
\begin{theorem} \label{32}
Let $\|\alpha\|_{\infty}=\sup \{|\alpha(x_1,x_2,$\dots$,x_k)|:(x_1,x_2,$\dots$, x_k)\in I^k\}$, and let $Id$ be the identity operator on $\mathcal{C}(I^k , \mathbb{R})$. For any $f \in \mathcal{C}(I^k, \mathbb{R})$, the perturbation error satisfies
$$
\left\|f_{\Delta, D}^\alpha-f\right\|_{\infty} \leq \frac{\|\alpha\|_{\infty}}{1-\|\alpha\|_{\infty}}\|f-D f\|_{\infty} .
$$
In particular, if $\alpha=0$, then $\mathcal{F}_{\Delta, D}^\alpha=I d$.
\end{theorem}
\begin{proof} For $(x_1,x_2,$\dots$,x_k) \in I^k $, from Equation (\ref{alphafrac}) we have 
\begin{equation}\label{nem123}
 \begin{aligned}
  &f_{\Delta, D}^\alpha(x_1,x_2,\dots,x_k)-f(x_1,x_2,\dots,x_k)\\&=\alpha(x_1,x_2,\dots,x_k)\left[f_{\Delta, D}^\alpha\left(v_{j,j_1}^{-1}(x_1),\ldots, v_{k,j_k}^{-1}(x_k)\right)\right.\\
  &~~~~ -(D f)\left.\left(v_{j,j_1}^{-1}(x_1),\ldots, v_{k,j_k}^{-1}(x_k)\right)\right].
   \end{aligned}
\end{equation}
Therefore, 
$$
\left|f_{\Delta, D}^\alpha(x_1,x_2,\dots,x_k)-f(x_1,x_2,\dots,x_k)\right| \leq\|\alpha\|_{\infty}\left\|f_{\Delta, D}^\alpha-L f\right\|_{\infty}.
$$
Since the above inequality is true for all $(x_1,\dots,x_k) \in \overset{k}{\underset{q=1}{ \prod}} I_{q,j_q},~ j_q \in  \Sigma_{N_q},~q \in   \Sigma_k.$\\
We conclude that
\begin{equation}\label{9}
\begin{aligned}
\left\|f_{\Delta, D}^\alpha-f\right\|_{\infty} \leq\|\alpha\|_{\infty}\left\|f_{\Delta, D}^\alpha-D f\right\|_{\infty} .
\end{aligned}
\end{equation}
Using the triangle inequality, we get
$$
\left\|f_{\Delta, D}^\alpha-f\right\|_{\infty} \leq\|\alpha\|_{\infty}\left[\left\|f_{\Delta, D}^\alpha-f\right\|_{\infty}+\|f-D f\|_{\infty}\right] .
$$
This demonstrates
$$
\left\|f_{\Delta, D}^\alpha-f\right\|_{\infty} \leq \frac{\|\alpha\|_{\infty}\|I d-D\|}{1-\|\alpha\|_{\infty}}\|f\|_{\infty} .
$$
For $\|\alpha\|_{\infty}=0$, the previous inequality produces $\left\|f_{\Delta, D}^\alpha-f\right\|_{\infty}=0$, and hence $f_{\Delta, D}^\alpha=f$ for all $f \in \mathcal{C}(I^k, \mathbb{R})$. That is, $\mathcal{F}_{\Delta, L}^\alpha=I d$.\\
    
\end{proof}
\begin{theorem} \label{33}               Let $\|\alpha\|_{\infty}=\sup \{|\alpha(x_1,x_2,$\dots$,x_k)|:(x_1,x_2,$\dots$, x_k)\in I^k\}$, and let $Id$ be the identity operator on $\mathcal{C}(I^k , \mathbb{R})$. The fractal operator $\mathcal{F}_{\Delta, D}^\alpha$ is a bounded linear operator with respect to the uniform norm on $\mathcal{C}(I^k, \mathbb{R})$. Furthermore, the operator norm satisfies
$$
\left\|\mathcal{F}_{\Delta, D}^\alpha\right\| \leq 1+\frac{\|\alpha\|_{\infty}\|I d-D\|}{1-\|\alpha\|_{\infty}}
$$   
\end{theorem}
\begin{proof}
 Let $f, g \in \mathcal{C}(I^k, \mathbb{R})$ and $\beta, \gamma \in \mathbb{R}$. For $(x_1,\dots,x_k) \in \overset{k}{\underset{q=1}{ \prod}} I_{q,j_q},~ j_q \in  \Sigma_{N_q},~q \in   \Sigma_k.$ \\
 We have
\begin{equation*}
\begin{aligned}
 \beta f_{\Delta, D}^\alpha(x_1,x_2,\dots,x_k)=&\beta f(x_1,x_2,\dots,x_k)\\&+\beta \alpha(x_1,x_2,\dots,x_k)\left[f_{\Delta, D}^\alpha\left(v_{j,j_1}^{-1}(x_1),\ldots, v_{k,j_k}^{-1}(x_k)\right)\right.\\
 &-(D f)\left.\left(v_{j,j_1}^{-1}(x_1),\ldots, v_{k,j_k}^{-1}(x_k)\right)\right],
 \end{aligned}
\end{equation*} 
\begin{equation*}
\begin{aligned}
\gamma g_{\Delta, D}^\alpha(x_1,x_2,\dots,x_k)=&\gamma g(x_1,x_2,\dots,x_k)\\&+\gamma \alpha(x_1,x_2,\dots,x_k)\left[g_{\Delta, D}^\alpha\left(v_{j,j_1}^{-1}(x_1),\ldots, v_{k,j_k}^{-1}(x_k)\right)\right.\\&-(D g)\left.\left(v_{j,j_1}^{-1}(x_1),\ldots, v_{k,j_k}^{-1}(x_k)\right)\right] .
\end{aligned}
\end{equation*}
Adding the above two equations, one gets
\begin{equation*}
\begin{aligned}
\left(\beta f_{\Delta, D}^\alpha+\gamma g_{\Delta, D}^\alpha\right)(x_1,x_2,\dots,x_k)= & (\beta f+\gamma g)(x_1,x_2,\dots,x_k)\\&+\alpha(x_1,x_2,\dots,x_k)
 \left(\beta f_{\Delta, D}^\alpha+\gamma g_{\Delta, D}^\alpha\right.\\
 &\left.-D(\beta f+\gamma g)\right))\left(v_{j,j_1}^{-1}(x_1),\ldots, v_{k,j_k}^{-1}(x_k)\right).
\end{aligned}
\end{equation*}
The previous equation reveals that $\beta f_{\Delta, D}^\alpha+\gamma g_{\Delta, D}^\alpha$ is the fixed point of the RB-operator
\begin{equation*}
\begin{aligned}
(T h)(x_1,x_2,\dots,x_k)= & \alpha(x_1,x_2,\dots,x_k) h\left(v_{j,j_1}^{-1}(x_1),\ldots, v_{k,j_k}^{-1}(x_k)\right)\\&+(\beta f+\gamma g)(x_1,x_2,\dots,x_k) \\
& -\alpha(x_1,x_2,\dots,x_k) L(\beta f+\gamma g)\left(v_{j,j_1}^{-1}(x_1),\ldots, v_{k,j_k}^{-1}(x_k)\right),
\end{aligned}
\end{equation*}
Since the fixed point of the RB operator is unique, we have $(\beta f+\gamma g)_{\Delta, D}^\alpha=$ $\beta f_{\Delta, D}^\alpha+\gamma g_{\Delta, D}^\alpha$, which reveals the linearity of the operator $\mathcal{F}_{\Delta, D}^\alpha$. From the previous item we write
$$
\left\|f_{\Delta, D}^\alpha\right\|_{\infty}-\|f\|_{\infty} \leq \frac{\|\alpha\|_{\infty}\|I d-D\|}{1-\|\alpha\|_{\infty}}\|f\|_{\infty} .
$$
It implies that
$$
\left\|\mathcal{F}_{\Delta, D}^\alpha(f)\right\|_{\infty} \leq\|f\|_{\infty}+\frac{\|\alpha\|_{\infty}\|I d-D\|}{1-\|\alpha\|_{\infty}}\|f\|_{\infty} .
$$
Therefore $\mathcal{F}_{\Delta, D}^\alpha$ is a bounded linear operator.\\   
\end{proof}
\begin{corollary} \label{34}          Let $\|\alpha\|_{\infty}=\sup \{|\alpha(x_1,x_2,$\dots$,x_k)|:(x_1,x_2,$\dots$, x_k)\in I^k\}$. For $\|\alpha\|_{\infty}<\|D\|^{-1}, \mathcal{F}_{\Delta, D}^\alpha$ is bounded below. In particular, $\mathcal{F}_{\Delta, D}^\alpha$ is one to one.
\end{corollary}
\begin{proof}
 From Equation (\ref{9})
\begin{equation*}
\begin{aligned}
\|f\|_{\infty}-\left\|f_{\Delta, D}^\alpha\right\|_{\infty} \leq\left\|f_{\Delta, D}^\alpha-f\right\|_{\infty} & \leq\|\alpha\|_{\infty}\left\|f_{\Delta, D}^\alpha-D f\right\|_{\infty} \\
& \leq\|\alpha\|_{\infty}\left(\left\|f_{\Delta, D}^\alpha\right\|_{\infty}+\|L\|\|f\|_{\infty}\right) .
\end{aligned}
\end{equation*}
Hence we get $\left(1-\|\alpha\|_{\infty}\|D\|\right)\|f\|_{\infty} \leq\left(1+\|\alpha\|_{\infty}\right)\left\|f_{\Delta, D}^\alpha\right\|_{\infty}$. If $\|\alpha\|_{\infty}<\|D\|^{-1}$, then
\begin{equation}\label{10}
    \begin{aligned}
     \|f\|_{\infty} \leq \frac{1+\|\alpha\|_{\infty}}{1-\|\alpha\|_{\infty}\|D\|}\left\|f_{\Delta, D}^\alpha\right\|_{\infty} . 
    \end{aligned}
\end{equation}

Thus $\mathcal{F}_{\Delta, D}^\alpha$ is bounded below.\\   
\end{proof}

\begin{lemma}\label{lem}
If $T$ is a bounded linear operator from a Banach space into itself such that $\|T\|< 1,$ then $(Id-T)^{-1}$ exists and it is bounded.
\end{lemma}
\begin{corollary}\label{36}        Let $\|\alpha\|_{\infty}=\sup \{|\alpha(x_1,x_2,$\dots$,x_k)|:(x_1,x_2,$\dots$, x_k)\in I^k\}$, and let $Id$ be the identity operator on $\mathcal{C}(I^k , \mathbb{R})$.
If $\|\alpha\|_{\infty}<(1+\|I d-D\|)^{-1}$, then $\mathcal{F}_{\Delta, D}^\alpha$ has a bounded inverse and consequently a topological automorphism (i.e., a bijective bounded linear map with a bounded inverse from $\mathcal{C}(I^k, \mathbb{R})$ to itself). Moreover,
$$
\left\|\left(\mathcal{F}_{\Delta, D}^\alpha\right)^{-1}\right\| \leq \frac{1+\|\alpha\|_{\infty}}{1-\|\alpha\|_{\infty}\|D\|} .
$$ 
\end{corollary} 
\begin{proof}
 From the hypothesis $\left\|I d-\mathcal{F}_{\Delta, D}^\alpha\right\| \leq \frac{\|\alpha\|_{\infty}\|I d-D\|}{1-\|\alpha\|_{\infty}}<1$. Consequently, Lemma(\ref{lem})  dictates that $\mathcal{F}_{\Delta, D}^\alpha$ has a bounded inverse. From equation(\ref{10}), we infer that
$$
\left\|\left(\mathcal{F}_{\Delta, D}^\alpha\right)^{-1}(f)\right\|_{\infty} \leq \frac{1+\|\alpha\|_{\infty}}{1-\|\alpha\|_{\infty}\|D\|}\|f\|_{\infty}.
$$
Which in turn yields the required bound for the operator norm of $\left(\mathcal{F}_{\Delta, D}^\alpha\right)^{-1}$.\\
\end{proof}
\begin{proposition} \label{37}                  Let $\|\alpha\|_{\infty}=\sup \{|\alpha(x_1,x_2,$\dots$,x_k)|:(x_1,x_2,$\dots$, x_k)\in I^k\}$. If $\|\alpha\|_{\infty} \neq 0$, then the fixed points of $L$ are the fixed points of the fractal operator $\mathcal{F}_{\Delta, D}^\alpha$ as well.
\end{proposition}
\begin{proof}
Let $\|\alpha\|_{\infty} \neq 0$, and $f$ be a fixed point of $L$. From equation(\ref{9})
$$
\left\|f_{\Delta, D}^\alpha-f\right\|_{\infty} \leq\|\alpha\|_{\infty}\left\|f_{\Delta, D}^\alpha-f\right\|_{\infty} .
$$
Since $\|\alpha\|_{\infty}<1$, this implies that $f_{\Delta, D}^\alpha=f$.\\
\end{proof}
 
\begin{lemma}\label{lema32}
    If $T$ is a bounded linear operator and $S$ is a compact operator on
a normed linear space $X$, then $T S$ and $ST$ are compact operators.
\end{lemma}
\begin{proposition}\label{39}
For $\|\alpha\|_{\infty}<\|D\|^{-1}$, the fractal operator $\mathcal{F}_{\Delta, D}^\alpha$ is not a compact operator.   
\end{proposition}


\begin{proof}

For $\|\alpha\|_{\infty}<\|D\|^{-1}$, we know that $\mathcal{F}_{\Delta, D}^\alpha: \mathcal{C}(I^k, \mathbb{R}) \rightarrow \mathcal{C}(I^k, \mathbb{R})$ is one-one. Note that the range space of $\mathcal{F}_{\Delta, D}^\alpha$ is infinite dimensional. We define the inverse map $\left(\mathcal{F}_{\Delta, D}^\alpha\right)^{-1}: \mathcal{F}_{\Delta, D}^\alpha(\mathcal{C}(I^k, \mathbb{R})) \rightarrow \mathcal{C}(I^k, \mathbb{R})$. With this choice of $\alpha, ~\mathcal{F}_{\Delta, D}^\alpha$ is bounded below, and hence it follows that $\left(\mathcal{F}_{\Delta, D}^\alpha\right)^{-1}$ is a bounded linear operator. Assume that $\mathcal{F}_{\Delta, D}^\alpha$ is a compact operator. Then by Lemma (\ref{lema32}), we deduce that the operator $T=\mathcal{F}_{\Delta, D}^\alpha\left(\mathcal{F}_{\Delta, D}^\alpha\right)^{-1}: \mathcal{F}_{\Delta, D}^\alpha(\mathcal{C}(I^k, \mathbb{R})) \rightarrow \mathcal{C}(I^k, \mathbb{R})$ is a compact operator, which is a contradiction to the infinite dimensionality of the space $\mathcal{F}_{\Delta, D}^\alpha(\mathcal{C}(I^k, \mathbb{R}))$. Therefore, $\mathcal{F}_{\Delta, D}^\alpha$ is not a compact operator.\\\\
\end{proof}
\begin{remark}
If $\alpha$ is a non-zero constant function in $I^k$ and $f$ is a fixed point of $\mathcal{F}_{\Delta, D}^\alpha$, then $f$ is a fixed point of $D$ as well. This can be easily seen as follows. Here $\alpha(x, y)=\alpha \neq 0$, a constant function on $I^k$. Now let $f$ be a fixed point of $\mathcal{F}_{\Delta, D}^\alpha$, that is, $f_{\Delta, D}^\alpha=f$. For $(x_1,x_2,\dots,x_k) \in I^k$, by the functional equation we have

\begin{equation*}
\begin{aligned}
f\left(v_{j,j_1}(x_1),\ldots, v_{k,j_k}(x_k)\right)=&f\left(v_{j,j_1}(x_1),\ldots, v_{k,j_k}(x_k)\right)\\&+\alpha[f(x_1,x_2,\dots,x_k)-L f(x_1,x_2,\dots,x_k)],
\end{aligned}
\end{equation*}
from which it follows that $Df=f$.
\end{remark}
\begin{theorem}

 Let $f \in \mathcal{C}(I^k, \mathbb{R})$.\\
\begin{enumerate}
    \item  If $\alpha_n \in \mathcal{C}(I^k, \mathbb{R})$ be such that $\left\|\alpha_n\right\|_{\infty}<1$ and $\alpha_n \rightarrow 0$ as $n \rightarrow \infty$. Then the corresponding sequence of $\alpha$-fractal functions $f_{\Delta, D}^{\alpha_n} \rightarrow f$ as $n \rightarrow \infty$.\\
\item If $D_n: \mathcal{C}(I^k, \mathbb{R}) \rightarrow \mathcal{C}(I^k, \mathbb{R})$ be a sequence of bounded linear operators such that $D_n f \neq f,~\left(D_nf\right)(x_{1,i_1},\ldots,x_{k,i_q})=f(x_{1,i_1},\ldots,x_{k,i_k}),$ for all $(i_1, \ldots,i_k) \in
\overset{k}{\underset{q=1}{ \prod}} \partial  \Sigma_{N_q,0}$  with respect to a net $\Delta$, and $D_n f \rightarrow f$. Then the corresponding sequence of $\alpha$-fractal functions $f_{\Delta, D_n}^\alpha \rightarrow f$ as $n \rightarrow \infty$ for any fixed admissible choice of the scale function $\alpha$.
\end{enumerate}
\end{theorem}
\begin{proof}
From $item (1)$ in the previous theorem, we note that the uniform error
bounds for the process of approximation of $f$ with $f_{\Delta, D_n}^{\alpha}$ is given by
$$
\left\|f_{\Delta, D_n}^\alpha-f\right\|_{\infty} \leq \frac{\|\alpha\|_{\infty}}{1-\|\alpha\|_{\infty}}\|f-D f\|_{\infty} .
$$
From this, the required results can be deduced.
\end{proof}

\begin{example}


We plot the $\alpha$- fractal function of the original function on the domain $[-1, 1]\times [-1,1]$. The 
      $\Delta $ is taken by the partition $\{-1, -0.5, 0, 0.5, 1\}$ of $[-1, 1]$. We choose the different scaling function $\alpha(x, y)$ corresponding to the original bivariate function $f(x,y)=exp(-x^2 - y^2) $ and base function $s(x,y)=x^2 y^2 f(x,y)$.

\begin{figure}
\includegraphics[height=9.5cm,width=13.5cm]{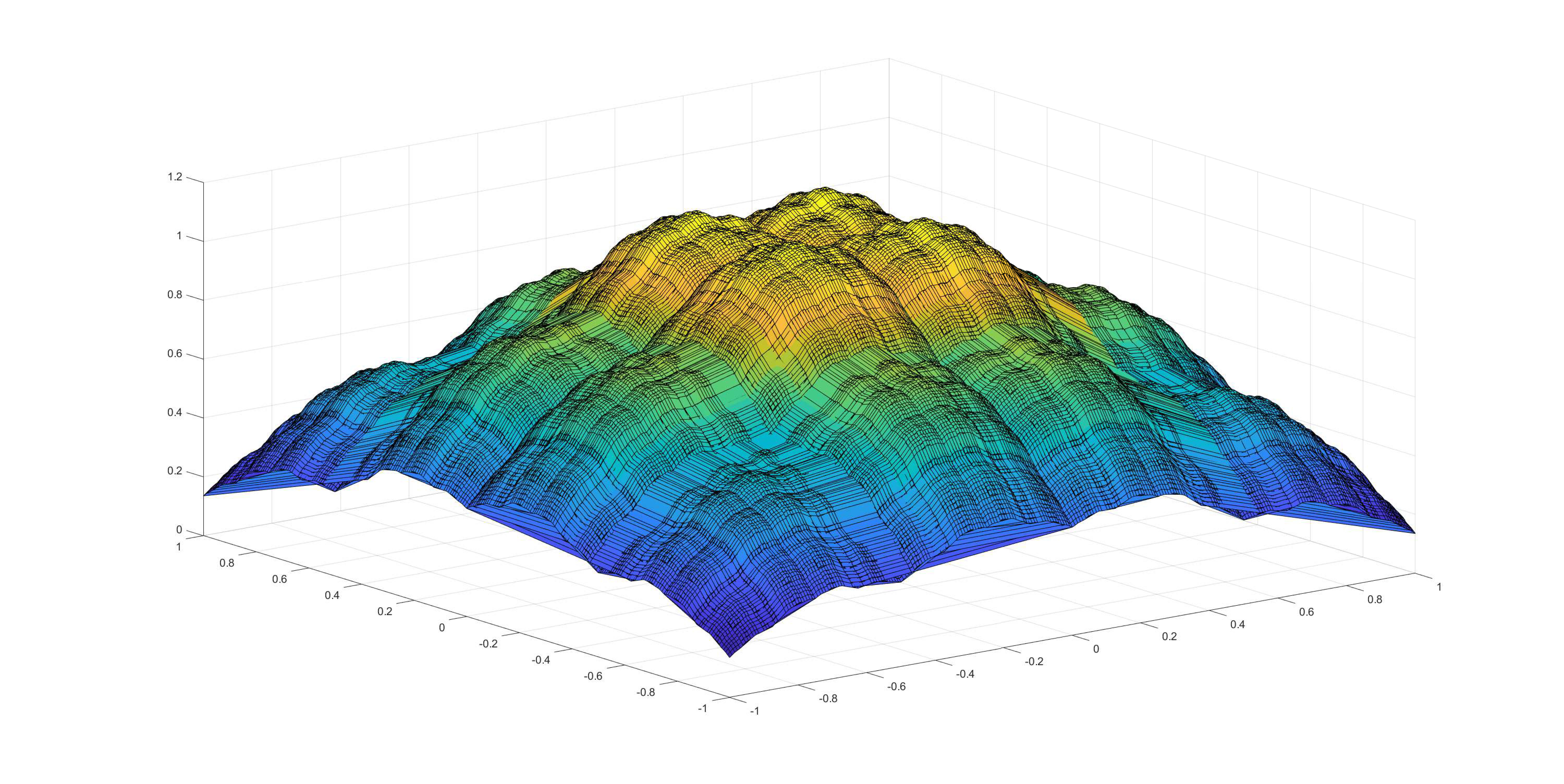}
		\caption{Fractal surface $f^{\alpha}$ when $\alpha=0.2, f(x,y)=exp(-x^2 - y^2), s(x,y)= x^2  y^2  f(x,y)$.}
		\label{Fig1}
\end{figure}

	\begin{figure}
\includegraphics[height=9.5cm,width=13.5cm]{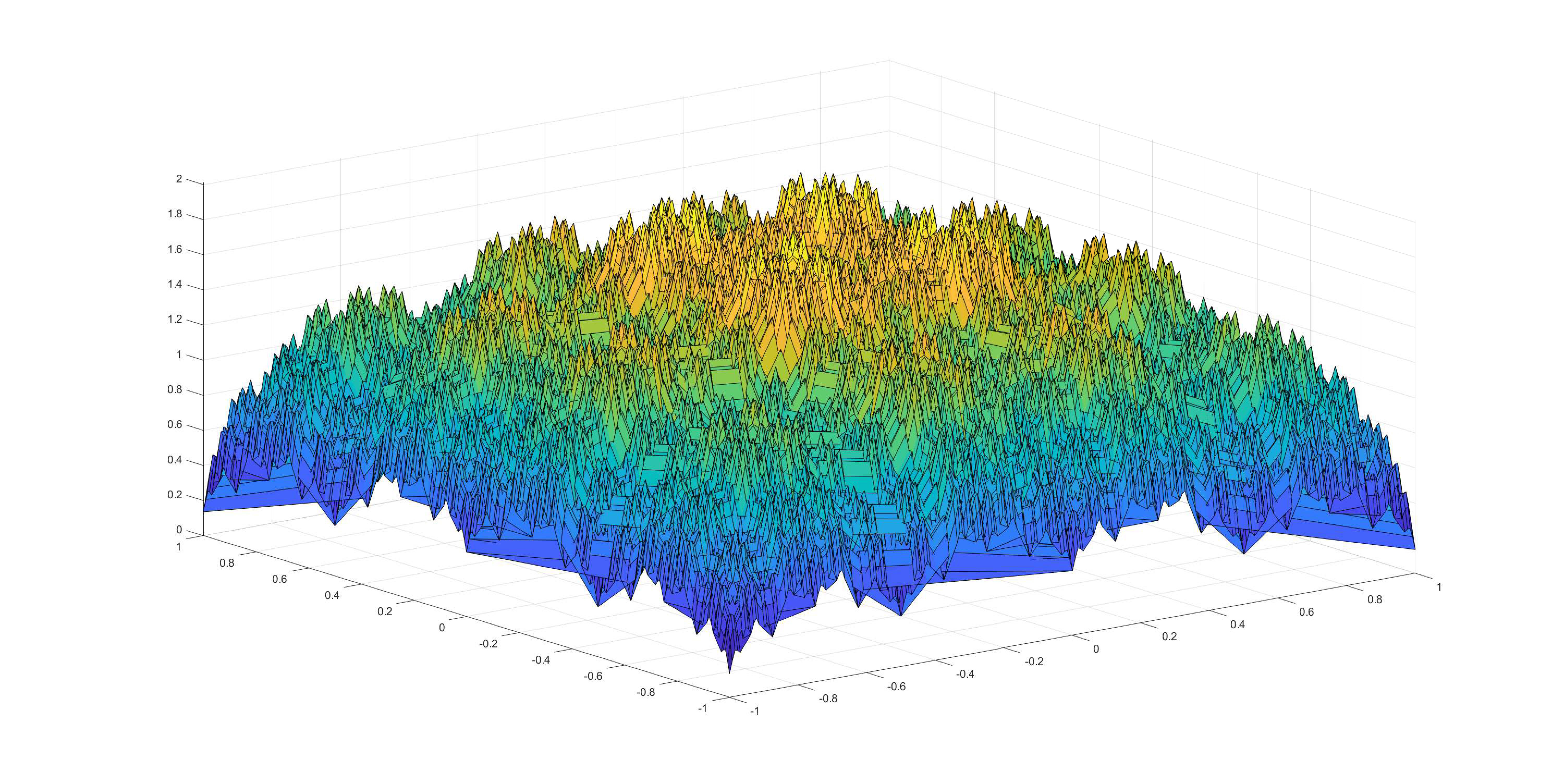}
		\caption{Fractal surface $f^{\alpha}$ when $\alpha=0.6, f(x,y)=exp(-x^2 - y^2), s(x,y)= x^2  y^2  f(x,y)$.}
		\label{Fig2}
\end{figure}

\end{example}

\begin{theorem}
 There exists a non-trivial closed invariant subspace for the fractal operator $\mathcal{F}_{\Delta, D}^\alpha: \mathcal{C}(I^k, \mathbb{R}) \rightarrow \mathcal{C}(I^k, \mathbb{R})$   
\end{theorem}
\begin{proof}
Fix a non-zero function $f\in \mathcal{C}(I^k, \mathbb{R})$ such that $f\left(x_1,x_2,\dots,x_k\right)=0$ for all $\left(x_1,x_2,\dots,x_k\right) \in \Delta$. Let $\left(\mathcal{F}_{\Delta, D}^\alpha\right)^r$ denote the $r$-fold autocomposition of $\mathcal{F}_{\Delta, D}^\alpha$ and $\left(\mathcal{F}_{\Delta, D}^\alpha\right)^0(f)=f$. Consider the $\mathcal{F}_{\Delta, D}^\alpha$-cyclic subspace generated by $f$, to wit,
$$
Y_f=\operatorname{span}\left\{f, \mathcal{F}_{\Delta, D}^\alpha(f),\left(\mathcal{F}_{\Delta, D}^\alpha\right)^2(f), \ldots\right\} .
$$
Clearly, $Y_f \neq\{0\}$ and $\mathcal{F}_{\Delta, D}^\alpha\left(Y_f\right) \subseteq Y_f$. Let $g \in Y_f$. Then, by the definition of $Y_f$, there exist constants $\beta_j \in \mathbb{R}$ such that
$$
g=\beta_1\left(\mathcal{F}_{\Delta, D}^\alpha\right)^{r_1}(f)+\beta_2\left(\mathcal{F}_{\Delta, D}^\alpha\right)^{r_2}(f)+\cdots+\beta_m\left(\mathcal{F}_{\Delta, D}^\alpha\right)^{r_m}(f),
$$
where $r_j \in \mathbb{N} \cup\{0\}$. By the interpolatory property of the fractal operator, we have
$$
f\left(x_1,x_2,\dots,x_k\right)=\left(\mathcal{F}_{\Delta, D}^\alpha(f)\right)\left(x_1,x_2,\dots,x_k\right), \forall\left(x_1,x_2,\dots,x_k\right) \in \Delta .
$$
Consequently, $g\left(x_1,x_2,\dots,x_k\right)=0$ for all $\left(x_1,x_2,\dots,x_k\right) \in \Delta$.
Take $M=\overline{Y_f}$. It is obvious that $\mathcal{F}_{\Delta, D}^\alpha(M) \subseteq M$, and hence that $M$ is a closed invariant subspace of $\mathcal{F}_{\Delta, D}^\alpha$

If $h \in M$, then there exists a sequence $\left(h_n\right)_{n \in \mathbb{N}} \subset Y_f$ such that $h_n \rightarrow h$ with respect to the supnorm. Since the uniform convergence (supnorm convergence) implies pointwise convergence, we assert that $h\left(x_1,x_2,\dots,x_k\right)=0$ for all $\left(x_1,x_2,\dots,x_k\right) \in \Delta$. Therefore a continuous function that is nonvanishing at some points in $\Delta$ is not an element in $M$. In particular, $M \neq \mathcal{C}(I^k, \mathbb{R})$, completing the proof.
The following remarks are straightforward, however worth recording.

\end{proof}

\section{Extension of Multivariate fractal operator to $\mathcal{L}^p(I^k, \mathbb{C})$ space}
In  this section, we will discuss the extension of multivariate fractal operator to $\mathcal{L}^p(I^k, \mathbb{C})$ spaces, $1 \leq p<\infty$. After this, we study that for a given $f \in \mathcal{L}^p(I^k, \mathbb{C})$, there exists a self-referential function $\bar{f}_{\Delta, D}^\alpha \in \mathcal{L}^p(I^k, \mathbb{C})$. 

As discuss in Lemma \ref{lem42} below, we define a multivariate fractal operator on the space of continuous complex valued functions on $I^k$, represented as $\mathcal{C}(I^k, \mathbb{C})$, induced with the $\mathcal{L}^p$-norm.

\begin{theorem}\label{41}
Let $\mathcal{C}(I^k, \mathbb{R})$ be induced with the $\mathcal{L}^p$-norm, $1 \leq p<\infty$, and $f \in \mathcal{C}(I^k, \mathbb{R})$. Next, suppose that $D$ is a linear map bounded corresponds to the $\mathcal{L}^p-$ norm on $\mathcal{C}(I^k, \mathbb{R})$. Then the below inequality occurs:
$$
\left\|f-f_{\Delta, D}^\alpha\right\|_{\mathcal{L}^p} \leq \frac{\|\alpha\|_{\infty}}{1-\|\alpha\|_{\infty}}\|f-D f\|_{\mathcal{L}^p}
$$
Consequently, the fractal operator $\mathcal{F}_{\Delta, D}^\alpha$ is bounded. 
\end{theorem} 
\begin{proof}
  Note that
$$
\begin{aligned}
\left\|f_{\Delta, D}^\alpha-f\right\|_{\mathcal{L}^p}^p & =\int_{I^k }\left|f_{\Delta, D}^\alpha(x_1,x_2,\dots,x_k)-f(x_1,x_2,\dots,x_k)\right|^p \mathrm{~d} x_1 \mathrm{~d} x_2\dots \mathrm{~d} x_k \\
& =\sum_{(1,\dots,k)} \int_{I_1\times \dots \times I_k}\left|f_{\Delta, D}^\alpha(x_1,x_2,\dots,x_k)-f(x_1,x_2,\dots,x_k)\right|^p \mathrm{~d} x_1 \mathrm{~d} x_2\dots \mathrm{~d} x_k.
\end{aligned}
$$
By using the functional equation for $f_{\Delta, D}^\alpha$ given in equation (\ref{alphafrac}) we get
\begin{equation*}
\begin{aligned}
\left\|f_{\Delta, D}^\alpha-f\right\|_{\mathcal{L}^p}^p  =&\sum_{(1,\dots,k)} \int_{I_1\times \dots \times I_k}\left|\alpha(x_1,x_2,\dots,x_k)\left[f_{\Delta, D}^\alpha\left(v_{j,j_1}^{-1}(x_1),\ldots, v_{k,j_k}^{-1}(x_k)\right)\right.\right.\\&-D f\left.\left.\left(v_{j,j_1}^{-1}(x_1),\ldots, v_{k,j_k}^{-1}(x_k)\right)\right]\right|^p \mathrm{~d} x_1 \mathrm{~d} x_2\dots \mathrm{~d} x_k \\
 \leq &\sum_{(1,\dots,k)} \int_{I_1\times \dots \times I_k}\|\alpha\|_{\infty}^p\left|f_{\Delta, D}^\alpha\left(v_{j,j_1}^{-1}(x_1),\ldots, v_{k,j_k}^{-1}(x_k)\right)\right.\\&-D f\left.\left(v_{j,j_1}^{-1}(x_1),\ldots, v_{k,j_k}^{-1}(x_k)\right)\right|^p \mathrm{~d} x_1 \mathrm{~d} x_2\dots \mathrm{~d} x_k .
\end{aligned}
\end{equation*}
Changing the variable $(x_1,x_2,\dots,x_k)$ to $(\tilde{x_1}, \dots,\tilde{x_k})$ through the transformations $\tilde{x_1}=v_{j,j_1}^{-1}(x_1)$,$\dots$, $\tilde{x_k}=v_{k,j_k}^{-1}(x_k)$ and using the change of variable formula for multiple integrals, we get
\begin{equation*}
\begin{aligned}
\left\|f_{\Delta, D}^\alpha-f\right\|_{\mathcal{L}^p}^p \leq & \sum_{(1,\dots,k)} \int_{I_1\times \dots \times I_k}\|\alpha\|_{\infty}^p\left|f_{\Delta, D}^\alpha(\tilde{x_1},\dots,\tilde{x_k})- D f(\tilde{x_1},\dots,\tilde{x_k})\right|^p\\& \left|\frac{\partial(x_1,\dots,x_k)}{\partial(\tilde{x_1},\dots,\tilde{x_k})}\right| \mathrm{d} \tilde{ x_1} \mathrm{~d} \tilde{x_2} \dots \mathrm{~d} \tilde{x_k}.
\end{aligned}
\end{equation*}
Therefore
\begin{equation}
    \begin{aligned}
     \left\|f_{\Delta, D}^\alpha-f\right\|_{\mathcal{L}^p}^p \leq\|\alpha\|_{\infty}^p\left\|f_{\Delta, D}^\alpha-D f\right\|_{\mathcal{L}^p}^p \sum_{j} \prod_{q=1}^k \left|a_{q,j_{q}} \right|   
    \end{aligned}
\end{equation}
implies
$$
\left\|f_{\Delta, D}^\alpha-f\right\|_{\mathcal{L}^p} \leq\|\alpha\|_{\infty}\left\|f_{\Delta, D}^\alpha-D f\right\|_{\mathcal{L}^p}
$$
By using this and the triangle inequality
$$
\left\|f_{\Delta, D}^\alpha-f\right\|_{\mathcal{L}^p} \leq\|\alpha\|_{\infty}\left[\left\|f_{\Delta, D}^\alpha-f\right\|_{\mathcal{L}^p}+\|f-D f\|_{\mathcal{L}^p}\right],
$$
hence
$$
\left\|f-f_{\Delta, D}^\alpha\right\|_{\mathcal{L}^p} \leq \frac{\|\alpha\|_{\infty}}{1-\|\alpha\|_{\infty}}\|f-D f\|_{\mathcal{L}^p}
$$
From the above bound for the perturbation error we get
$$
\left\|f_{\Delta, D}^\alpha\right\|_{\mathcal{L}^p}-\|f\|_{\mathcal{L}^p} \leq\left\|f-f_{\Delta, D}^\alpha\right\|_{\mathcal{L}^p} \leq \frac{\|\alpha\|_{\infty}}{1-\|\alpha\|_{\infty}}\|f- D f\|_{\mathcal{L}^p},
$$
using which we infer that
$$
\left\|\mathcal{F}_{\Delta, D}^\alpha(f)\right\|_{\mathcal{L}^p} \leq\left[1+\frac{\|\alpha\|_{\infty}}{1-\|\alpha\|_{\infty}}\|I d-D\|\right]\|f\|_{\mathcal{L}^p}
$$
That is, $\mathcal{F}_{\Delta, D}^\alpha$ be a bounded operator.
\end{proof} 

\begin{lemma}\label{lem42}
Suppose $\mathcal{F}^\alpha$ be a fractal operator on $\mathcal{C}(I^k, \mathbb{R})$,induced with the $\mathcal{L}^p$-norm. The operator $\mathcal{F}_{\mathbb{C}}^\alpha: \mathcal{C}(I^k, \mathbb{C}) \rightarrow \mathcal{C}(I^k, \mathbb{C})$ characterized by
$$
\mathcal{F}_{\mathbb{C}}^\alpha(f)=\mathcal{F}_{\mathbb{C}}^\alpha\left(f_1+i f_2\right)=\mathcal{F}^\alpha\left(f_1\right)+i \mathcal{F}^\alpha\left(f_2\right)
$$ is a bounded linear operator.    
\end{lemma} 
\begin{proof}
Since $\mathcal{F}^\alpha$ and $\mathcal{F}_{\mathbb{C}}^\alpha$ is linear. So we have to show that $\mathcal{F}_{\mathbb{C}}^\alpha$ is a bounded operator. To show this, we note that

\begin{align}\label{namesomething}
\left\|\mathcal{F}_{\mathbb{C}}^\alpha(f)\right\|_{\mathcal{L}^p}^p & =\int_{I ^k}\left|\mathcal{F}_{\mathbb{C}}^\alpha(f)\right|^p \mathrm{~d} x_1 \mathrm{~d} x_2\dots \mathrm{~d} x_k \nonumber\\
& =\int_{I^k}\left[\left|\mathcal{F}^\alpha\left(f_1\right)\right|^2+\left|\mathcal{F}^\alpha\left(f_2\right)\right|^2\right]^{\frac{p}{2}} \mathrm{~d} x_1 \mathrm{~d} x_2\dots \mathrm{~d} x_k \nonumber\\
& \leq 2^{\frac{p}{2}} \int_{I ^k}\left[\left|\mathcal{F}^\alpha\left(f_1\right)\right|^p+\left|\mathcal{F}^\alpha\left(f_2\right)\right|^p\right] \mathrm{d} x_1 \mathrm{~d} x_2\dots \mathrm{~d} x_k \nonumber\\
&=2^{\frac{p}{2}}\left[\left\|\mathcal{F}^\alpha\left(f_1\right)\right\|_{\mathcal{L}^p}^p+\left\|\mathcal{F}^\alpha\left(f_2\right)\right\|_{\mathcal{L}^p}^p\right] \\
& \leq 2^{\frac{p}{2}}\left\|\mathcal{F}^\alpha\right\|^p\left[\left\|f_1\right\|_{\mathcal{L}^p}^p+\left\|f_2\right\|_{\mathcal{L}^p}^p\right] \nonumber\\
& =2^{\frac{p}{2}}\left\|\mathcal{F}^\alpha\right\|^p\left[\int_{I^k}\left|f_1\right|^p \mathrm{~d}x_1 \mathrm{~d} x_2\dots \mathrm{~d} x_k +\int_{I^k}\left|f_2\right|^p \mathrm{~d} x_1 \mathrm{~d} x_2\dots \mathrm{~d} x_k\right] \nonumber\\
& \leq 2^{\frac{p}{2}+1}\left\|\mathcal{F}^\alpha\right\|^p \int_{I^k}|f|^p \mathrm{~d}x_1 \mathrm{~d} x_2\dots \mathrm{~d} x_k \nonumber\\
& =2^{\frac{p}{2}+1}\left\|\mathcal{F}^\alpha\right\|^p\|f\|_{\mathcal{L}^p}^p,\nonumber
\end{align}

where the first inequality in the preceding analysis is acquired by taking $( a+b)^r \leq 2^r\left(a^r+b^r\right)$, and the rest are plain. Consequently, 
$$
\left\|\mathcal{F}_{\mathbb{C}}^\alpha(f)\right\|_{\mathcal{L}^p} \leq 2^{\frac{1}{2}+\frac{1}{p}}\left\|\mathcal{F}^\alpha\right\|\|f\|_{\mathcal{L}^p},
$$
proving that $\mathcal{F}_{\mathcal{C}}^\alpha$ be a bounded operator and $\left\|\mathcal{F}_{\mathbb{C}}^\alpha\right\| \leq 2^{\frac{1}{2}+\frac{1}{p}}\left\|\mathcal{F}^\alpha\right\|$.
\end{proof}
\begin{remark} \label{43}
  For $p \geq 2$, using the inequality
$$
(a+b)^r \leq 2^{r-1}\left(a^r+b^r\right), \quad a \geq 0, b \geq 0, r \geq 1.
$$
similar computations as in \eqref{namesomething} provide $\left\|\mathcal{F}_{\mathcal{C}}^\alpha\right\| \leq 2^{\frac{1}{2}-\frac{1}{p}}\left\|\mathcal{F}^\alpha\right\|$. In particular, for $p=2$, we have $\left\|\mathcal{F}_{\mathbb{C}}^\alpha\right\| \leq\left\|\mathcal{F}^\alpha\right\|$. 
\end{remark}
Let us remind the below basic theorem.
\begin{theorem}
  If an operator $T: Z \rightarrow Y$ is linear and bounded, $Y$ is a Banach space and $Z$ is dense in $X$, then $T$ can be extended to $X$ preserving the norm of $T$.  
\end{theorem} 
In the sequel, we fix the notation $M:=2^{\frac{1}{2}+\frac{1}{p}}$.

\begin{theorem}
Let $1 \leq p<\infty, \mathcal{C}(I^k, \mathbb{R})$ be equipped with 
 $\mathcal{L}_p$-norm and $\mathcal{F}_{\mathcal{C}}^\alpha$ : $\mathcal{C}(I^k, \mathbb{C}) \rightarrow \mathcal{C}(I^k, \mathbb{C})$ be the fractal operator defined in the previous lemma. Then there exists a bounded linear operator $\overline{\mathcal{F}}_{\mathbb{C}}^\alpha: \mathcal{L}^p(I^k, \mathbb{C}) \rightarrow \mathcal{L}^p(I^k, \mathbb{C})$ such that its restriction to $\mathcal{C}(I^k, \mathbb{C})$ is $\mathcal{F}_{\mathbb{C}}^\alpha$ and $\left\|\mathcal{F}^\alpha\right\| \leq\left\|\overline{\mathcal{F}}_{\mathbb{C}}^\alpha\right\|=\left\|\mathcal{F}_{\mathbb{C}}^\alpha\right\| \leq$ $M\left\|\mathcal{F}^\alpha\right\|$
\end{theorem}
\begin{proof}
It is well-known that $\mathcal{C}(I^k, \mathbb{C})$ is dense in $\mathcal{L}^p(I^k, \mathbb{C})$, for $1 \leq p<\infty$. From the previous lemma, we have a bounded linear operator $\mathcal{F}_{\mathbb{C}}^\alpha: \mathcal{C}(I^k, \mathbb{C}) \rightarrow \mathcal{C}(I^k, \mathbb{C})$. Now using the previous theorem we conclude that there exists a bounded linear operator $\overline{\mathcal{F}}_{\mathbb{C}}^\alpha: \mathcal{L}^p(I^k, \mathbb{C}) \rightarrow \mathcal{L}^p(I^k, \mathbb{C})$ such that
$$
\left\|\overline{\mathcal{F}}_{\mathbb{C}}^\alpha\right\|=\left\|\mathcal{F}_{\mathbb{C}}^\alpha\right\| \leq M\left\|\mathcal{F}^\alpha\right\| .
$$
Furthermore, the previous theorem gives $\overline{\mathcal{F}}_{\mathrm{C}}^\alpha(f)=\mathcal{F}_{\mathbb{C}}^\alpha(f)=\mathcal{F}^\alpha(f), \forall f \in$ $\mathcal{C}(I^k, \mathbb{R})$. Hence
$$
\begin{aligned}
\left\|\mathcal{F}^\alpha\right\| & :=\sup \left\{\left\|\mathcal{F}^\alpha(f)\right\|_{\mathcal{L}^p}: f \in \mathcal{C}(I^k, \mathbb{R}),\|f\|_{\mathcal{L}^p}=1\right\} \\
& =\sup \left\{\left\|\overline{\mathcal{F}}_{\mathbb{C}}^\alpha(f)\right\|_{\mathcal{L}^p}: f \in C(I^k, \mathbb{R}),\|f\|_{\mathcal{L}^p}=1\right\}
\end{aligned}
$$
This implies that $\left\|\mathcal{F}^\alpha\right\| \leq\left\|\overline{\mathcal{F}}_{\mathbb{C}}^\alpha\right\|$ and hence that $\left\|\mathcal{F}^\alpha\right\| \leq\left\|\overline{\mathcal{F}}_{\mathbb{C}}^\alpha\right\| \leq M\left\|\mathcal{F}^\alpha\right\|$. 
\end{proof} 
\begin{remark}
   In view of Remark (\ref{43}), for $p=2$, we have $\left\|\overline{\mathcal{F}}_{\mathrm{C}}^\alpha\right\|=\left\|\mathcal{F}^\alpha\right\|$; for the univariate counterpart see \cite{M1}. 
\end{remark} 

\begin{remark}
  Similar to Remark (\ref{43}), we note that the constant $M$ appearing in the previous theorem is not claimed to be optimal.  
\end{remark} 

\begin{lemma}
 Let $L: \mathcal{C}(I^k, \mathbb{R}) \rightarrow \mathcal{C}(I^k, \mathbb{R})$ be a bounded linear operator with respect to the $\mathcal{L}^p$-norm on $\mathcal{C}(I^k, \mathbb{R}), 1 \leq p<\infty$. Then there exists a bounded linear operator $\bar{D}_{\mathbb{C}}: \mathcal{L}^p(I^k, \mathbb{C}) \rightarrow \mathcal{L}^p(I^k, \mathbb{C})$ such that the restriction of $\bar{D}_{\mathbb{C}}$ to $\mathcal{C}(I^k, \mathbb{R})$ is $D$ and $\|D\| \leq\left\|\bar{D}_{\mathbb{C}}\right\| \leq M\|D\|$. Moreover, the extension of $I d-D$ is $I d-\bar{D}_C$.   
\end{lemma} 

\begin{proof}
 The first part of the lemma can be proved similar to Lemma (\ref{lem42}) via $D_{\mathbb{C}}$. For the rest, let $f \in \mathcal{L}^p(I^k, \mathbb{C})$ and a sequence of complex-valued continuous functions such that $f_n \rightarrow f$ with respect to the $\mathcal{L}^p$-norm. Then
$$
(I d-D)_{\mathbb{C}} f_n=f_n-D_{\mathbb{C}} f_n \rightarrow f-\bar{D}_{\mathbb{C}} f
$$
and hence $\overline{(I d-D)_{\mathbb{C}}}=I d-\bar{D}_{\mathbb{C}}$.  
\end{proof}
\begin{lemma}
For each $f \in \mathcal{C}(I^k, \mathbb{C})$,
$$
\left\|\mathcal{F}_{\mathbb{C}}^\alpha(f)-f\right\|_{\mathcal{L}^p} \leq M\|\alpha\|_{\infty}\left\|\mathcal{F}_{\mathbb{C}}^\alpha(f)-D_{\mathbb{C}}(f)\right\|_{\mathcal{L}^p} .
$$ 
\end{lemma} 
\begin{proof}
Let $f \in \mathcal{C}(I^k, \mathbb{C})$, where $f=f_1+i f_2$ for some $f_1, f_2 \in \mathcal{C}(I^k, \mathbb{R})$. Since $\mathcal{F}_{\mathbb{C}}^\alpha(f):=\mathcal{F}^\alpha\left(f_1\right)+i \mathcal{F}^\alpha\left(f_2\right)$, we have
$$
\operatorname{Re}\left(\mathcal{F}_{\mathbb{C}}^\alpha(f)\right)=\mathcal{F}^\alpha\left(f_1\right), \quad \operatorname{Im}\left(\mathcal{F}_{\mathbb{C}}^\alpha(f)\right)=\mathcal{F}^\alpha\left(f_2\right)
$$
Furthermore, Using Theorem (\ref{41}) and some basic inequalities, we have
$$
\begin{aligned}
\left\|\mathcal{F}_{\mathbb{C}}^\alpha(f)-f\right\|_{\mathcal{L}^p}^p= & \int_{I^k}\left|\mathcal{F}_{\mathbb{C}}^\alpha(f)-f\right|^p \mathrm{~d} x_1 \mathrm{~d} x_2\dots \mathrm{~d} x_k \\
= & \int_{I^k}\left[\left|\mathcal{F}^\alpha\left(f_1\right)-f_1\right|^2+\left|\mathcal{F}^\alpha\left(f_2\right)-f_2\right|^2\right]^{\frac{p}{2}} \mathrm{~d}x_1 \mathrm{~d} x_2\dots \mathrm{~d} x_k \\
\leq & 2^{\frac{p}{2}} \int_{I^k}\left[\left|\mathcal{F}^\alpha\left(f_1\right)-f_1\right|^p+\left|\mathcal{F}^\alpha\left(f_2\right)-f_2\right|^p\right] \mathrm{d} x_1 \mathrm{~d} x_2\dots \mathrm{~d} x_k \\
= & 2^{\frac{p}{2}}\left[\left\|\mathcal{F}^\alpha\left(f_1\right)-f_1\right\|_{\mathcal{L}^p}^p+\left\|\mathcal{F}^\alpha\left(f_2\right)-f_2\right\|_{\mathcal{L}^p}^p\right] \\
\leq & 2^{\frac{p}{2}}\|\alpha\|_{\infty}^p\left[\left\|\mathcal{F}^\alpha\left(f_1\right)-D\left(f_1\right)\right\|_{\mathcal{L}^p}^p+\left\|\mathcal{F}^\alpha\left(f_2\right)-D\left(f_2\right)\right\|_{\mathcal{L}^p}^p\right] \\
= & 2^{\frac{p}{2}}\|\alpha\|_{\infty}^p\left[\int_{I^p}\left|\mathcal{F}^\alpha\left(f_1\right)-D\left(f_1\right)\right|^p \mathrm{~d} x_1 \mathrm{~d} x_2\dots \mathrm{~d} x_k \right. \\
& \left.+\int_{I^k}\left|\mathcal{F}^\alpha\left(f_2\right)-D\left(f_2\right)\right|^p \mathrm{~d} x_1 \mathrm{~d} x_2\dots \mathrm{~d} x_k\right] \\
\leq & 2^{\frac{p}{2}+1}\|\alpha\|_{\infty}^p \int_{I^k}\left|\mathcal{F}_{\mathbb{C}}^\alpha(f)-D_{\mathbb{C}}(f)\right|^p \mathrm{~d} x_1 \mathrm{~d} x_2\dots \mathrm{~d} x_k\\
= & 2^{\frac{p}{2}+1}\|\alpha\|_{\infty}^p\left\|\mathcal{F}_{\mathcal{C}}^\alpha(f)-D_{\mathbb{C}}(f)\right\|_{\mathcal{L}^p}^p,
\end{aligned}
$$
hence the proof.
\end{proof} 
The proof of the next corollary follows from the way in which $\overline{\mathcal{F}}_{\mathbb{C}}^\alpha$ is defined using the denseness of $\mathcal{C}(I^k, \mathbb{C})$ in $\mathcal{L}^p(I^k, \mathbb{C})$.
\begin{corollary}
    
For each $f \in \mathcal{L}^p(I^k, \mathbb{C})$,
$$
\left\|\overline{\mathcal{F}}_{\mathbb{C}}^\alpha(f)-f\right\|_{\mathcal{L}^p} \leq M\|\alpha\|_{\infty}\left\|\overline{\mathcal{F}}_{\mathbb{C}}^\alpha(f)-\bar{D}_{\mathbb{C}} f\right\|_{\mathcal{L}^p}
$$
\end{corollary} 
Using the previous lemma, we deduce the next theorem in a similar way as that in  (\ref{32},\ref{33},\ref{34},\ref{36},\ref{37},\ref{39}). We avoid the proof, however, recall that for a bounded linear operator $T: X \rightarrow X$ in a Hilbert space $X$, the following orthogonal decomposition holds:
$$
X=\overline{\text { Range }(T)} \oplus \operatorname{Ker}\left(T^*\right) .
$$
\begin{theorem}
Let $\|\alpha\|_{\infty}=\sup \{|\alpha(x, y)|:(x, y) \in I^k\}, M\|\alpha\|_{\infty}<1$ and let $I d$ be the identity operator on $\mathcal{L}^p(I^k, \mathbb{C})$.\\
\begin{itemize}
    \item For any $f \in \mathcal{L}^p(I^k, \mathbb{C})$, the perturbation error satisfies
$$
\left\|\overline{\mathcal{F}}_{\mathrm{C}}^\alpha(f)-f\right\|_{\mathcal{L}^p} \leq \frac{M^2\|\alpha\|_{\infty}\|I d-D\|_{\mathcal{L}^p}}{1-M\|\alpha\|_{\infty}}\|f\|_{\mathcal{L}^p} .
$$
In particular, If $\|\alpha\|_{\infty}=0$ then $\overline{\mathcal{F}}_{\mathrm{C}}^\alpha=Id$.\\
\item The norm of the fractal operator $\overline{\mathcal{F}}_{\mathrm{C}}^\alpha: \mathcal{L}^p(I^k, \mathbb{C}) \rightarrow \mathcal{L}^p(I^k, \mathbb{C})$ satisfies
$$
\left\|\overline{\mathcal{F}}_{\mathbb{C}}^\alpha\right\|_{\mathcal{L}^p} \leq 1+\frac{M^2\|\alpha\|_{\infty}\|I d-D\|}{1-M\|\alpha\|_{\infty}} .
$$
\item For $\|\alpha\|_{\infty}<M^{-2}\|D\|^{-1}, \overline{\mathcal{F}}_{\mathrm{C}}^\alpha$ is bounded below. In particular, $\overline{\mathcal{F}}_{\mathrm{C}}^\alpha$ is one to one and the range of $\overline{\mathcal{F}}_{\mathrm{C}}^\alpha$ is closed.
\item $\|\alpha\|_{\infty}<\left(M+M^2\|I d-D\|\right)^{-1}$, then $\overline{\mathcal{F}}_{\mathrm{C}}^\alpha$ has a bounded inverse. Moreover,
$$
\left\|\left(\overline{\mathcal{F}}_{\mathrm{C}}^\alpha\right)^{-1}\right\| \leq \frac{1+M\|\alpha\|_{\infty}}{1-M^2\|\alpha\|_{\infty}\|D\|} .
$$
\item If $\|\alpha\|_{\infty}<\left(M+M^2\|I d-D\|\right)^{-1}, \overline{\mathcal{F}}_{\mathrm{C}}^\alpha$ is Fredholm and its index is 0 .
\item Let $p=2$ and $\|\alpha\|_{\infty}<\|D\|^{-1}$, then $\mathcal{L}^2(I^k, \mathbb{C})=\operatorname{Range}\left(\overline{\mathcal{F}}_{\mathrm{C}}^\alpha\right) \oplus$ $\operatorname{Ker}\left(\left(\overline{\mathcal{F}}_{\mathrm{C}}^\alpha\right)^*\right)$ where $\left(\overline{\mathcal{F}}_{\mathrm{C}}^\alpha\right)^*$ is the adjoint operator of $\overline{\mathcal{F}}_{\mathrm{C}}^\alpha$.
\end{itemize}
\end{theorem} 
\section{Some Approximation Aspects}
In this section we shall return to the multivariate $\alpha$-fractal functions in the function space $\mathcal{C}(I^k, \mathbb{R})$. First let us recall the following well-known definition.
\begin{definition}
A Schauder basis in an infinite dimensional Banach space $X$ is a sequence $\left(e_n\right)$ of elements in $X$ satisfying the following condition: for every $x$ in $X$, there is a unique sequence $\left(a_n(x)\right)$ of scalars such that
$$
x=\sum_{n=1}^{\infty} a_n(x) e_n, \quad \text { i.e., } \quad\left\|x-\sum_{n=1}^m a_n(x) e_n\right\| \rightarrow 0 \quad \text { as } m \rightarrow \infty .
$$
The coefficients $a_n(x)$ are linear functions of $x$ uniquely determined by the basis referred to as the associated sequence of coefficient functionals.  
\end{definition}

Now, we find a Schauder basis for $\mathcal{C}(I^k, \mathbb{R})$ consisting of fractal functions; the maps involved are perturbations of those belonging to a classical basis for $\mathcal{C}(I^k, \mathbb{R})$. The central idea is to use the fact that a topological automorphism preserves a Schauder basis, however, we provide the details in the following.

\begin{theorem}
There exists a Schauder basis consisting of multivariate fractal functions for the space $\mathcal{C}\left([0,1]^k, \mathbb{R}\right)$.
\end{theorem}
\begin{proof}
Let $\left(e_n\right)$ be a Schauder basis of $\mathcal{C}\left([0,1]^k, \mathbb{R}\right)$, whose existence is hinted at the last paragraph. Choose $\alpha$ such that $\|\alpha\|_{\infty}<(1+\|I d-D\|)^{-1}$, so that by Corollary (\ref{36}), the fractal operator $\mathcal{F}_{\Delta, D}^\alpha$ is a topological automorphism. If $g \in \mathcal{C}\left([0,1]^k, \mathbb{R}\right)$ then $\left(\mathcal{F}_{\Delta, D}^\alpha\right)^{-1}(g) \in \mathcal{C}\left([0,1]^k, \mathbb{R}\right)$, so that
$$
\left(\mathcal{F}_{\Delta, D}^\alpha\right)^{-1}(g)=\sum_{n=1}^{\infty} a_n\left(\left(\mathcal{F}_{\Delta, D}^\alpha\right)^{-1}(g)\right) e_n
$$
By the continuity of the fractal linear operator $\mathcal{F}_{\Delta, D}^\alpha$ it follows that
$$
g=\mathcal{F}_{\Delta, D}^\alpha\left(\mathcal{F}_{\Delta, D}^\alpha\right)^{-1}(g)=\sum_{n=1}^{\infty} a_n\left(\left(\mathcal{F}_{\Delta, D}^\alpha\right)^{-1}(g)\right) e_n^\alpha
$$
where $e_n^\alpha=\mathcal{F}_{\Delta, D}^\alpha\left(e_n\right)$. Assume that $g=\sum_{n=1}^{\infty} b_n e_n^\alpha$ was another representation of $g$. Since $\left(\mathcal{F}_{\Delta, D}^\alpha\right)^{-1}$ is also continuous, we have
$$
\left(\mathcal{F}_{\Delta, D}^\alpha\right)^{-1}(g)=\sum_{n=1}^{\infty} b_n e_n
$$
and hence that $b_n=a_n\left(\left(\mathcal{F}_{\Delta, D}^\alpha\right)^{-1}(g)\right)$ for each $n$. Consequently, $\left(e_n^\alpha\right)$ is a Schauder basis for $\mathcal{C}\left([0,1]^k, \mathbb{R}\right)$, obtaining the desired conclusion.
\end{proof}
\begin{definition}
Consider the fractal operator $\mathcal{F}_{\Delta, D}^\alpha: \mathcal{C}(I^k, \mathbb{R}) \rightarrow \mathcal{C}(I^k, \mathbb{R})$ defined by $f \mapsto f_{\Delta, D}^\alpha$; see Sect. 3 . Let $p \in \mathcal{C}(I^k, \mathbb{R})$ be a multivariate polynomial. Then $\mathcal{F}_{\Delta, D}^\alpha(p)=p_{\Delta, D}^\alpha$, denoted for simplicity by $p^\alpha$, is referred to as a multivariate fractal polynomial; see also \cite{M2}. Let $\mathcal{P}(I^k) \subset \mathcal{C}(I^k, \mathbb{R})$ be the space of all multivariate polynomials, then we denote by $\mathcal{P}^\alpha(I^k)$, the image space $\mathcal{F}_{\Delta, D}^\alpha(\mathcal{P}(I^k))$
  
\end{definition} 
Let $\mathcal{P}_{m_1,\dots,m_q}(I^q)$ be the set of all multivariate polynomials of total degree at most $m_1+\dots+m_q$ defined on $I^q$. That is,

\begin{equation*}
\begin{aligned}
\mathcal{P}_{m_1.\dots.m_q}(I^q)=\Big\{p(x_1,\dots,x_q)=\sum_{i_1=0}^{m_1} &\dots.\sum_{i_q=0}^{m_q} a_{i_1,\dots,i_q} x_1^{i_1}.\dots.x_q^{i_q}: a_{i j} \in \mathbb{R},\\& 0 \leq i \leq m \text { and } 0 \leq j \leq n \Big\}.
\end{aligned}
\end{equation*}
We let $\mathcal{P}_{m_1,\dots,m_q}^\alpha(I^q)=\mathcal{F}_{\Delta, D}^\alpha\left(\mathcal{P}_{m_1,\dots,m_q}(I^q)\right)$

\begin{theorem}
Let $\mathcal{C}(I^k, \mathbb{R})$ be endowed with the uniform norm, $f \in \mathcal{C}(I^k, \mathbb{R})$, and $D: \mathcal{C}(I^k, \mathbb{R}) \rightarrow \mathcal{C}(I^k, \mathbb{R}), D \neq I d$ be a bounded linear operator satisfying $(Df)(x_{1,i_1},\ldots,x_{k,i_k})=f(x_{1,i_1},\ldots,x_{k,i_k}),$ for all $(i_1, \ldots,i_k) \in
\overset{k}{\underset{q=1}{ \prod}} \partial  \Sigma_{N_q,0}$. For any $\epsilon>0$, net $\Delta$ of the rectangle $I^k$, there exists a bivariate fractal polynomial $p^\alpha$ such that
$$
\left\|f-p^\alpha\right\|_{\infty}<\epsilon $$ .
\end{theorem}
\begin{proof}
 Let $\epsilon>0$ be given. By the Stone-Weierstrass theorem, there exists a polynomial function $p$ in two variables such that
$$
\|f-p\|_{\infty}<\frac{\epsilon}{2}.
$$
Fix a net $\Delta$ of the rectangle $I^k$, a bounded linear operator $D: \mathcal{C}(I^k, \mathbb{R}) \rightarrow \mathcal{C}(I^k, \mathbb{R}), D \neq I d$ satisfying $(Df)(x_{1,i_1},\ldots,x_{k,i_k})=f(x_{1,i_1},\ldots,x_{k,i_k}),$ for all $(i_1, \ldots,i_k) \in
\overset{k}{\underset{q=1}{ \prod}} \partial  \Sigma_{N_q,0}$. Choose $\alpha: I^k \rightarrow \mathbb{R}$ as continuous function on $I^k$ with $\|\alpha\|_{\infty}=\sup \{|\alpha(x_1,x_2,\dots,x_k)|:(x_1,x_2,\dots,x_k) \in I^k\}<1$ such that
$$
\|\alpha\|_{\infty}<\frac{\frac{\epsilon}{2}}{\frac{\epsilon}{2}+\|I d-D\|\|p\|_{\infty}}.
$$
Then we have
$$
\begin{aligned}
\left\|f-p^\alpha\right\|_{\infty} & \leq\|f-p\|_{\infty}+\left\|p-p^\alpha\right\|_{\infty} \\
& \leq\|f-p\|_{\infty}+\frac{\|\alpha\|_{\infty}}{1-\|\alpha\|_{\infty}}\|I d-D\|\|p\|_{\infty} \\
& <\frac{\epsilon}{2}+\frac{\epsilon}{2} \\
& =\epsilon .
\end{aligned}
$$
In the above, the first inequality is just the triangle inequality, second follows from Theorem (\ref{32}) and third is obvious.   
\end{proof}

\subsection*{Acknowledgements}
 The work of first author  is financially supported by the
CSIR, India with grant no:
 09/1028(0019)/2020-EMR-I. 
 \subsection*{Conflict of interest} The authors declare that they have no conflict of interest.
 \subsection*{Data availability}  Data sharing not applicable to this article as no data sets were generated or analysed during the current study.

\bibliographystyle{amsplain}

\begin{thebibliography}{30}

\bibitem{EPJST} V. Agrawal, T. Som, Fractal dimension of $\alpha$-fractal function on the Sierpi\'nski Gasket, Eur. Phys. J. Spec. Top. , (2021), https://doi.org/10.1140/epjs/s11734-021-00304-9.
\bibitem{RIM} V. Agrawal, T. Som, $L^{p}$ approximation using fractal functions on the Sierpi\'nski gasket, Results Math 77, 74, (2022), https://doi.org/10.1007/s00025-021-01565-5.
\bibitem{Vishal} V. Agrawal, T. Som, S. Verma, On bivariate fractal approximation, The Journal of Analysis, (2022), 1-19.
\bibitem{VSV1} V. Agrawal, T. Som, S. Verma, A note on stability and fractal dimension of bivariate alpha-fractal functions, Numerical Algorithms (2023).
\bibitem {MB} M. F. Barnsley, Fractals everywhere, Academic Press, Orlando, Florida, (1988).
\bibitem{MFB1} M. F. Barnsley, Super Fractals. Cambridge University Press, Cambridge (2006)
\bibitem{MFB2} M. F. Barnsley, J.H., Hardin, D.P.: Recurrent iterated function systems.
Constr. Approx. 5(1), 3–31 (1989)
\bibitem{Amit} Amit, Vineeta Basotia, and Ajay Prajapati, Non-stationary $\phi$-contractions and associated fractals,  J Anal (2022) https://doi.org/10.1007/s41478-022-00518-7.
\bibitem{SS1} S. Chandra, S. Abbas, The calculus of fractal interpolation surfaces, Fractals 29(3), (2021).
\bibitem{SS4}  S. Chandra, S. Abbas, Analysis of mixed Weyl-Marchaud fractional derivative and box dimensions, Fractals 29(06), (2021).
\bibitem{SS2} S. Chandra, S. Abbas, Analysis of fractal dimension of mixed Riemann-Liouville integral, Numerical Algorithms, (2022), https://doi.org/10.1007/s11075-022-01290-2.
\bibitem{SS3} S. Chandra, S. Abbas, Box dimension of mixed Katugampola fractional integral of two-dimensional continuous functions, Fract Calc Appl Anal, (2022), https://doi.org/10.1007/s13540-022-00050-2.
\bibitem{SS5} S. Chandra, S. Abbas, On fractal dimensions of fractal functions using functions spaces, Bull. Aust. Math. Soc. (2022).

\bibitem{YF} Y. Fisher, Fractal Image Compression: Theory and Application. Springer, New York (1995)

\bibitem{Hu} J.E. Hutchinson, Fractals and self-similarity. Indiana Univ. Math. J. 30, 713–747 (1981)
\bibitem{Jha2} S. Jha, S. Verma, A study on fractal operator corresponding to non-stationary fractal interpolation functions, In Frontiers of Fractal Analysis Recent Advances and Challenges, (2022), 50-66.
\bibitem{JVC1} S. Jha, S. Verma, A. K. B. Chand, Non-stationary zipper $\alpha$-fractal functions and associated fractal operator, Fractional Calculus and Applied Analysis (2022) 1-26.
\bibitem{Jha} S. Jha, S. Verma, Dimensional analysis of  $\alpha$ -fractal functions, Results in Mathematics 76 (4),  (2021), 1-24.

\bibitem{Liang} Y. S. Liang,
Box dimensions of Riemann–Liouville fractional integrals of continuous functions of bounded variation, Nonlinear Analysis: Theory, Methods $\&$ Applications, 72(11), (2010), 4304–4306.
\bibitem{Mas1} P. R. Massopust, Fractal functions and their applications. Chaos Solitons Fractals 8(2) 171–190 (1997).
\bibitem{Mas2} P. R. Massopust, Fractal Functions, Fractal Surfaces, and Wavelets. 2nd. ed., Academic Press, 2016.




\bibitem {M2} M. A. Navascu\'es, Fractal polynomial interpolation, Z. Anal. Anwend. 25(2) (2005) 401-418.
\bibitem {M1} M. A. Navascu\'es, Fractal approximation, Complex Anal. Oper. Theory 4(4) (2010) 953-974.
\bibitem{Mnew} M. A. Navascu\'es, New equilibria of non-autonomous discrete dynamical systems, Chaos Solitons Fractals 152 (2021),  111413.
\bibitem{Mnew2} M. A. Navascu\'es, S. Verma, Non-stationary $\alpha$-fractal surfaces, Mediterranean Journal of Mathematics (2023) 20:48.
\bibitem{PAS3}  V. Agrawal, M. Pandey, T. Som, Box Dimension and Fractional Integrals of Multivariate $\alpha$-Fractal Functions, Mediterr J. Math. (2023).
\bibitem{PSV2} Megha Pandey, Tanmoy Som, and Saurabh Verma, Set-valued $\alpha$-fractal functions, arXiv preprint arXiv:2207.02635 (2022).
\bibitem{PV1} Pata V. Fixed point theorems and applications. Cham: Springer, (2019).
\bibitem{PV11} S. A. Prasad, S. Verma, Fractal interpolation functions on products of the Sierpinski gaskets, Chaos, Solitons and Fractals 166 (2023) 112988.
\bibitem{SR1} S. Ri, A new fixed point theorem in the fractal space, Indagationes Mathematicae, (2015).

\bibitem{SR2} S. Ri, A new idea to construct the fractal interpolation
function , Indagationes Mathematicae, (2018).
\bibitem{Ruan} Huo-Jun Ruan and Qiang Xu, Fractal interpolation surfaces on Rectangular Grids, Bull. Aust. Math. Soc., 91 (2015)  435-446.
\bibitem{SP} A. Sahu, A. Priyadarshi, On the box-counting dimension of Graphs of harmonic functions on the Sierpi\'{n}ski gasket, J. Math. Anal. Appl. 487, (2020). 


\bibitem{VM1} S. Verma, P. R. Massopust, Dimension preserving approximationn (2022), Aequationes Mathematicae, https://doi.org/10.48550/arXiv.2002.05061.
\bibitem{VL1} S. Verma, Y. S. Liang, Effect of the Riemann-Liouville fractional integral on unbounded variation points,(2020), https://doi.org/10.48550/arXiv.2008.11113.
\bibitem{VFS}  S. Verma, A. Sahu, Bounded variation on the Sierpi\'nski Gasket, Fractals, (2022), DOI: 10.1142/S0218348X2250147X.

\bibitem{sverma3} S. Verma, P. Viswanathan, Bivariate functions of bounded variation: Fractal dimension and fractional integral, Indagationes Mathematicae 31 (2020) 294-309.

\bibitem{sverma5} S. Verma, P. Viswanathan, A Fractal Operator Associated with Bivariate Fractal Interpolation Functions on Rectangular Grids, Results Math 75, 28 (2020).

\bibitem{Ver71} M. Verma, A. Priyadarshi, S. Verma, Fractal dimensions of fractal transformations and Quantization dimensions for bi-Lipschitz mappings, (2022) https://doi.org/10.48550/arXiv.2212.09669.

\bibitem{Ver1} M. Verma, A. Priyadarshi, S. Verma, Vector-valued fractal functions: fractal dimension and fractional calculus, (2022), https://doi.org/10.48550/arXiv.2205.00892 .
\bibitem{Ver12}  M. Verma, A. Priyadarshi, S. Verma, Fractal dimension for a class of complex-valued fractal interpolation functions, (2022), https://doi.org/10.48550/arXiv.2204.03622.
\end{thebibliography}

\end{document}